\documentclass[]{interact}

\usepackage{epstopdf}
\usepackage[caption=false]{subfig}

\usepackage{xcolor}

\usepackage[numbers,sort&compress]{natbib}
\bibpunct[, ]{[}{]}{,}{n}{,}{,}
\renewcommand\bibfont{\fontsize{10}{12}\selectfont}
\makeatletter
\def\NAT@def@citea{\def\@citea{\NAT@separator}}
\makeatother

\theoremstyle{plain}
\newtheorem{theorem}{Theorem}[section]
\newtheorem{lemma}[theorem]{Lemma}
\newtheorem{corollary}[theorem]{Corollary}

\theoremstyle{definition}

\newtheorem{example}[theorem]{Example}

\theoremstyle{remark}
\newtheorem{remark}{Remark}

\newcommand{\la}{\left \langle}
\newcommand{\ra}{\right \rangle}
\newcommand{\q}{\sqrt{1 - q^2}}
\newcommand{\Tx}{\langle Tx, x \rangle}
\newcommand{\Txy}{\langle Tx, y \rangle}
\newcommand{\Txz}{\langle Tx, z \rangle}

\newcommand{\h}{\mathcal{H}_1}
\newcommand{\hh}{\mathcal{H}_2}
\newcommand{\hhh}{\mathcal{H}_1 \oplus \mathcal{H}_2}

\newcommand{\norm}[1]{\left\lVert#1\right\rVert}

\allowdisplaybreaks

\begin{document}


\title{New upper bounds for the $q$-numerical radius of Hilbert space operators}

\author{
\name{Arnab Patra\textsuperscript{a}\thanks{CONTACT Arnab Patra. Email: arnabp@iitbhilai.ac.in} and Falguni Roy\textsuperscript{b}}
\affil{\textsuperscript{a}Department of Mathematics, Indian Institute of Technology Bhilai, GEC campus, Raipur, India 492015; \textsuperscript{b}Department of Mathematical and Computational Sciences, National Institute of Technology Karnataka, Surathkal, India 575025}
}

\maketitle

\begin{abstract}
This article introduces several new upper bounds for the $q$-numerical radius of bounded linear operators on complex Hilbert spaces. Our results refine some of the existing upper bounds in this field. The $q$-numerical radius inequalities of products and commutators of operators follow as special cases. Finally, some new inequalities for the $q$-numerical radius of $2 \times 2$ operator matrices are established. 
\end{abstract}

\begin{amscode}
Primary 47A12; 47A30\\
Secondary 15A60  
\end{amscode}

\section{Introduction}
Let $\mathcal{H}$ denotes a complex Hilbert space with the inner product $\la .,. \ra$ and $\mathcal{B(H)}$ denotes the $C^*$-algebra of bounded linear operators on $\mathcal{H}$. For $T \in \mathcal{B(H)},$ the operator norm of $T$ can be defined as
\[\|T\| = \sup_{\|x\| = 1} \|Tx\|.\]
Another expression of $\|T\|$ in terms of the inner product is as follows
\[\|T\| = \sup_{\|x\| = \|y\| = 1} |\la Tx, y \ra|.\]
A norm $|||.|||$ on $\mathcal{B(H)}$ is said to be a unitarily invariant norm if it satisfies $|||UTV||| = |||T|||$ for all $T \in \mathcal{B(H)}$ and for all unitary operators $U$ and $V$ in $\mathcal{B(H)}.$ A norm $|||.|||$ on $\mathcal{B(H)}$ is said to be a weakly unitarily invariant norm if it satisfies $|||UTU^*||| = |||T|||$ for all $T \in \mathcal{B(H)}$ and for all unitary operators $U$ in $\mathcal{B(H)}.$

The numerical range of $T$, which is denoted by $W(T)$, is defined by
\[W(T) = \{\la Tx,x \ra : x \in \mathcal{H}, \  \|x\| = 1\}.\]
The most important properties of $W(T)$ are that it always forms a convex set and its closure contains the spectrum of $T.$ The numerical radius $\omega(T)$ and the Crawford number $c(T)$ of $T \in \mathcal{B(H)}$ are defined by
\begin{eqnarray*}
 \omega(T) &=& \sup_{\|x\| = 1} |\la Tx, x \ra|,\\
 c(T) &=& \inf_{\|x\| = 1} |\la Tx, x \ra|.  
\end{eqnarray*}
The numerical radius $\omega(T)$ defines a weakly unitarily invariant norm in $\mathcal{B(H)}$. The operator norm and numerical radius are both equivalent which follows from the following well-known inequality
\begin{equation} \label{int1}
    \frac{\|T\|}{2} \leq \omega(T) \leq \|T\|.
\end{equation}
The above inequalities are sharp. Equality holds in the first inequality if $T^2 = 0$ and in the second inequality if $T$ is a normal operator. For the past several years researchers have attempted to refine the above inequality. Kittaneh \cite{kittaneh2003numerical,kittaneh2005numerical} proved, respectively, that, if $T \in \mathcal{B(H)}$, then
\begin{equation} \label{int2}
    \omega(T) \leq \frac{1}{2} \||T| + |T^*|\| \leq \frac{1}{2} \left(\|T\| + \sqrt{\|T^2\|} \right),
\end{equation}
\begin{equation} \label{int3}
     \frac{1}{4} \|T^*T + TT^*\| \leq \omega^2(T) \le \frac{1}{2} \| T^*T + TT^* \|, 
\end{equation}
and if $A, B, C, D, S, T \in \mathcal{B(H)}$ then
\begin{equation} \label{eq8}
    \omega(ATB + CSD) \leq \frac{1}{2} \| A|T^*|^{2(1 - \alpha)} A^* + B^* |T|^{2 \alpha} B + C |S^*|^{2(1 - \alpha)} C^* + D^* |S|^{2 \alpha} D \|,
\end{equation}
for all $\alpha \in [0,1]$ where $|T| = (T^*T)^{\frac{1}{2}},$ the absolute value of $T.$ Later on these bounds were refined extensively. For a detailed review of the numerical radius inequalities, we refer to the book \cite{bhunia2022lectures}.

There are several generalizations of the classical numerical range exist in the literature. Our focus will be on the $q$-numerical range and its radius of an operator. Let $T \in \mathcal{B(H)}$ and $q \in [0,1].$ The $q$-numerical range $W_q(T)$ and $q$-numerical radius $\omega_q(T)$ of $T$ are defined respectively as
\begin{eqnarray*}
    W_q(T) &=& \{\la Tx, y \ra : x,y \in \mathcal{H}, \ \|x\| = \|y\| = 1, \ \la x,y \ra = q\},\\
    \omega_q(T) &=& \sup_{z \in W_q(T)} |z|.
\end{eqnarray*}
It is easy to verify that if $q = 1$ then $W_q(T)$ reduces to the classical numerical range $W(T).$ The set $W_q(T)$ was first introduced by Marcus and Andresen \cite{marcus1977constrained} in 1977 for a linear transformation $T$ defined over an $n$-dimensional unitary space. Nam-Kiu Tsing \cite{tsing1984constrained} established the convexity of the $q$-numerical range. Several properties of $W_q(T)$ are discussed by Li et al. \cite{li1994generalized} and Li and Nakazato \cite{li1998some}. Chien and Nakazato \cite{chien2002davis} described the boundary of the $q$-numerical range of a square matrix  using the concept of the Davis-Wieldant shell. The $q$-numerical range of shift operators is also studied \cite{chien2007q,chien2012numerical}. Duan \cite{duan2009perfect} draws attention to the vital significance that the idea of $q$-numerical range plays in characterizing the perfect distinguishability of quantum operations. Recently, Moghaddam et al. \cite{fakhri2022q} have studied several $q$-numerical radius bounds. The following are a few of the inequalities they have derived. 
\begin{eqnarray}
   && \frac{q}{2(2 - q^2)} \|T\| \leq \omega_q(T) \leq \|T\|,\label{oth1}\\
   && \omega^2_q(T) \leq \frac{q^2}{4} \left(\|T\| + \sqrt{\|T^2\|}\right)^2 + (1 - q^2 + 2q \q) \| |T| \|^2 ,\label{oth2}\\
   && \frac{q^2}{4(2 - q^2)^2} \| T^*T + TT^* \| \leq \omega^2(T) \leq \frac{q^2}{2(1 - \q)^2} \| T^*T + TT^* \label{oth3}\|.
\end{eqnarray}
These results are in fact generalizations of the corresponding inequalities in (\ref{int1}), (\ref{int2}), and (\ref{int3}) respectively for numerical radius. 

Our interest in this paper lies in the direction of obtaining refined $q$-numerical radius inequalities. In section 2 we established upper bounds for $q$-numerical radius that generalize the results of \cite{fakhri2022q}. In addition, the $q$-numerical radius bounds for $2 \times 2$ operator matrices are also discussed in section 3. Several examples with figures are provided to supplement the results.

\section{$q$-numerical radius of $T \in \mathcal{B(H)}$}

First, we record a few important properties of the $q$-numerical radius in the following lemmas.

\begin{lemma} \cite[P. 380]{gau2021numerical}
    Let $T \in \mathcal{B(H)}$ and $q \in [0,1]$, then
    \begin{enumerate}
        \item[(i)] if $\dim \mathcal{H} = 1$ then $W_q(T)$ is non-empty if and only if $q = 1$ and for $\dim \mathcal{H} \geq 2,$ $W_q(T)$ is always non-empty,
        \item[(ii)] $W_q(T)$ is a bounded subset of $\mathbb{C}$ and it is compact if $\mathcal{H}$ is finite-dimensional,
        \item[(iii)] $W_q(U^*TU) = W_q(T)$ for any unitary operator $U \in \mathcal{B(H)},$
        \item[(iv)] $W_q(aT + bI) = a W_q(T) + bq$ for complex numbers $a$ and $b,$
        \item[(v)] $W_{\lambda q}(T) = \lambda W_q(T)$ for any complex numeber $\lambda$ with $|\lambda| = 1,$
        \item[(vi)] $W_q(T)^* = W_q(T)^* = \{\overline{z} : z \in W_q(T)\},$
        \item[(vii)] $q \sigma(T) \subseteq \overline{W_q(T)}.$
    \end{enumerate}
\end{lemma}

\begin{lemma} \cite{fakhri2022q}
    The $q$-numerical radius defines a semi-norm on $\mathcal{B(H)}.$
\end{lemma}

\begin{lemma} \cite[Proposition 2.11]{li1994generalized}
    Let $T \in \mathcal{B(H)}$ and $m(T) = \min \{ \| T - \lambda I\| : \lambda \in \mathbb{C}\}$ then
    \[ \mbox{either } W_0(T) = \{ z : |z| < m(T) \} \mbox{ or, } W_0(T) = \{ z : |z| \leq m(T) \}.\]
\end{lemma}
The number $m(T)$ is known as the transcendental radius of $T$. Stampfli \cite{stampfli1970norm} proved that there exists a unique complex number $\mu \in \overline{W(T)}$ such that
\[m(T) =  \min \{ \| T - \lambda I\| : \lambda \in \mathbb{C}\} = \| T - \mu I \|.\]
Prasanna \cite{prasanna1981norm} derived another expression for $m(T)$ which is 
\[m^2(T) = \sup_{\|x\| = 1} (\|Tx\|^2 - |\la Tx,x \ra|^2).\]

For our study, the following results are crucial. 

\begin{lemma}{(Bessel's Inequality)}
Let $\mathcal{E}$ be a orthonormal set in $\mathcal{H}$ and $h \in \mathcal{H},$ then
\[\sum_{e \in \mathcal{E}} |\langle h, e \rangle|^2 \leq \|h\|^2.\]
\end{lemma}

\begin{lemma} \cite{furuta1986simplified} \label{lemma1}
    If $T \in \mathcal{B(H)},$ then
    \[|\la Tx,y \ra|^2 \leq \la |T|^{2 \alpha} x, x \ra \la |T^*|^{2 (1 - \alpha)}y,y \ra,\]
    for all $x,y \in \mathcal{H}$ and $\alpha \in [0,1].$
\end{lemma}

Now we are ready to prove the $q$-numerical radius inequalities.

\begin{theorem} \label{th2}
Let $T \in \mathcal{B(H)}$ and $q\in [0,1],$ then
\begin{equation} \label{eq3}
    \omega_q^2(T) \leq q^2 \omega^2(T) + ( 1 - q^2 + q\sqrt{1 - q^2}) \|T\|^2.
\end{equation}
\end{theorem}

\begin{proof}
For $ q=1$ the inequality holds trivially. Let $q \in [0,1)$ and $x, y \in \mathcal{H}$ such that $\|x\| = 1 = \|y\|$ with $\langle x,y \rangle = q.$ Then $y$ can be expressed as $y = q x + \sqrt{1 - q^2}z,$ where $\|z\| = 1$ and $\la x,z \ra = 0.$ In this setting we have
\begin{equation} \label{eq1}
  |\langle Tx, y \rangle | \leq q |\langle Tx, x \rangle| + \sqrt{1 - q^2}|\langle Tx, z \rangle |.  
\end{equation}
Let $\mathcal{E}$ be any orthonormal set in $\mathcal{H}$ containing $x$ and $z.$ From Bessel's Inequality it follows
\begin{eqnarray*}
&&\sum_{e \in \mathcal{E} \setminus \{x\}} |\la Tx, e \ra|^2 + |\la Tx, x \ra|^2 \leq \|Tx\|^2 \\
& \Rightarrow & |\Txz|^2 \leq \sum_{e \in \mathcal{E} \setminus \{x\}} |\la Tx, e \ra|^2 \leq \|Tx\|^2 - |\la Tx, x \ra|^2.
\end{eqnarray*}
Using the above relation in equation (\ref{eq1}), we get

  \begin{eqnarray} \label{array1}
     |\langle Tx, y \rangle | & \leq & q |\langle Tx, x \rangle| + \sqrt{1 - q^2} (\|Tx\|^2 - |\Tx|^2)^{\frac{1}{2}}\\
    \Rightarrow  |\Txy|^2  & \leq & \left( q |\langle Tx, x \rangle| + \sqrt{1 - q^2} (\|Tx\|^2 - |\Tx|^2)^{\frac{1}{2}}  \right)^2 \nonumber\\
     & = & q^2 |\Tx|^2 + 2q \q |\Tx| (\|Tx\|^2 - |\Tx|^2)^{\frac{1}{2}} \nonumber \\
     && + (1 - q^2) (\|Tx\|^2 - |\Tx|^2) \nonumber \\
     & \leq & q^2 |\Tx|^2 + 2q \q |\Tx| (\|Tx\|^2 - |\Tx|^2)^{\frac{1}{2}} \nonumber \\
     && + (1 - q^2) \|Tx\|^2 \nonumber \\
     & \leq & q^2 |\Tx|^2 + q \q (|\Tx|^2 + \|Tx\|^2 - |\Tx|^2 ) \nonumber \\
     && + (1 - q^2) \|Tx\|^2 \nonumber \\
     & = & q^2 |\Tx|^2 + (1 - q^2 + q \q) \|Tx\|^2 \nonumber \\
     & \leq & q^2 \omega^2(T) + ( 1 - q^2 + q\sqrt{1 - q^2}) \|T\|^2. \nonumber
  \end{eqnarray}
  Taking supremum for all $x,y \in \mathcal{H}$ with $\|x\| = \|y\| = 1$ and $\la x,y \ra = q$ we get 
  \[
    \omega_q^2(T) \leq q^2 \omega^2(T) + ( 1 - q^2 + q\sqrt{1 - q^2}) \|T\|^2.
\]
\end{proof}

Now we mention a few observations and derive several corollaries based on the above Theorem. 

\begin{enumerate}


\item From the result $\omega(T) \leq \frac{1}{2} (\|T\| + \|T^2\|^\frac{1}{2}),$ mentioned in Theorem 1 of \cite{kittaneh2003numerical}, we have the following corollary from Theorem \ref{th2}.

\begin{corollary} \label{cor2}
    For $T \in \mathcal{B(H)}$ and $q \in [0,1]$ we have
\begin{equation} \label{eq5}
    \omega_q^2(T) \leq \frac{q^2}{4} (\|T\| + \|T^2\|^\frac{1}{2})^2 + ( 1 - q^2 + q\sqrt{1 - q^2}) \|T\|^2.
\end{equation}
\end{corollary}
Using the fact $\| |T| \| = \| T \|,$ we have the following relations
\begin{eqnarray*}
  \omega_q^2(T) & \leq & \frac{q^2}{4} (\|T\| + \|T^2\|^\frac{1}{2})^2 + ( 1 - q^2 + q\sqrt{1 - q^2}) \||T|\|^2 \\
   & \leq &  \frac{q^2}{4} (\|T\| + \|T^2\|^\frac{1}{2})^2 + ( 1 - q^2 + 2q\sqrt{1 - q^2}) \|T\|^2. 
\end{eqnarray*}
This shows that the inequality mentioned in (\ref{eq5}) provides an improvement on the result (\ref{oth2}), proved in Theorem 2.10 in \cite{fakhri2022q}.\\



\item From the relation $\omega^2(T) \leq \frac{1}{2} \|T^*T+TT^*\|$ as obtained in Theorem 1 in \cite{kittaneh2005numerical} we have the following corollary follows from the Theorem \ref{th2}.

\begin{corollary} \label{cor1}
    For $T \in \mathcal{B(H)}$ and $q \in [0,1],$ the following relation holds
   \begin{eqnarray} \label{eq7}
       \omega_q^2(T) \leq \frac{q^2}{2} \|T^*T+TT^*\| + (1 - q^2 + q \q) \|T\|^2. 
    \end{eqnarray} 
\end{corollary}

Here we prove the refinement of the above result in comparison to the existing upper bound (\ref{oth3}) of $\omega^2(T)$ mentioned in Theorem 3.1 in \cite{fakhri2022q}. For this, we use the fact that 
\[\|T^*T + TT^*\| \geq \|T\|^2.\]
From Corollary \ref{cor1} it follows
\begin{eqnarray*}
\omega_q^2(T) & \leq & \frac{q^2}{2} \|T^*T+TT^*\| + (1 - q^2 + q \q) \|T\|^2 \\
& \leq & \left( 1 - \frac{q^2}{2} + q \q) \right) \|T^*T+TT^*\|\\
& \leq &  \frac{q^2}{2(1- \q)^2}   \|T^*T+TT^*\|.
\end{eqnarray*}
The last inequality follows from the fact that
\begin{eqnarray*}
  &&\frac{q^2}{2(1- \q)^2} -  \left( 1 - \frac{q^2}{2} + q \q) \right) \\ 
  & = & \frac{(1 - q^2)(2 - q^2) + 2 \q (1 - q^3)}{2q^2} \geq 0.
\end{eqnarray*}
\end{enumerate}
\begin{remark}
 It is crucial to note that the upper bound given in equation (\ref{oth3}) goes unbounded when $q$ approaches zero. However, the upper bound stated in Corollary \ref{cor1} is applicable to all $q\in [0,1]$.    
\end{remark}

Now we provide a few examples to demonstrate our results.
For this the following lemma is crucial.
\begin{lemma}\label{lemma4}(\cite{nakazato1994c})
    Suppose $0\leq q \leq 1$ and $T\in M_{2}(\mathbb{C})$. Then $T$ is unitarily similar to $e^{it}\begin{pmatrix}\gamma & a \\ b & \gamma \end{pmatrix}$ for some $0\leq t \leq 2\pi$ and $0\leq b \leq a$. Also,
    \begin{eqnarray*}
        W_q(T)=e^{it}\{\gamma q + r((c+pd)\cos{s}+i(d+pc)\sin{s}): 0 \leq r \leq 1, 0 \leq s \leq 2\pi\},
    \end{eqnarray*}
    with $c=\frac{a+b}{2},d=\frac{a-b}{2}$ and $p=\sqrt{1-q^2}$.
\end{lemma}

\begin{example}\label{example1}
Consider the matrix $T=\begin{pmatrix}0 & \frac{1}{35}\\ 0 & 0\end{pmatrix}$. Then $\|T\|=\frac{1}{35}$. From Lemma \ref{lemma4}, 
\begin{equation*}
    W_q(T)=\left\{\frac{re^{is}}{70}(1+\sqrt{1-q^2}:0\leq r \leq 1,0\leq s \leq 2\pi)\right\}.
\end{equation*}
Therefore, $\omega_q(T)=\frac{1}{70}(1+\sqrt{1-q^2})$. In this example, we will compare the upper bounds (\ref{oth2}) and (\ref{oth3}) of $\omega_q(T)$ obtained in \cite{fakhri2022q} with our results Corollary \ref{cor2} and Corollary \ref{cor1}. 
Since $\|T^2\|=0$, from equations (\ref{oth2}) and (\ref{eq5}) we repectively get 
\begin{eqnarray}\label{eq17}
    \omega_q(T) \leq \frac{1}{35}\sqrt{1-\frac{3q^2}{4}+2q\sqrt{1-q^2}}
\end{eqnarray} and

\begin{eqnarray}\label{eq18}
    \omega_q(T) \leq \frac{1}{35}\sqrt{1-\frac{3q^2}{4}+q\sqrt{1-q^2}}.
\end{eqnarray}
Figure \ref{fig1} presents the graphical representation of upper bounds (\ref{eq17}), (\ref{eq18}) and $\omega_q(T)$. 

\begin{figure}[h]
\caption{Comparision of $\omega_q(T)$ with upper bounds (\ref{eq17}) and (\ref{eq18}) for Example \ref{example1}}
\centering
\includegraphics[scale=.90]{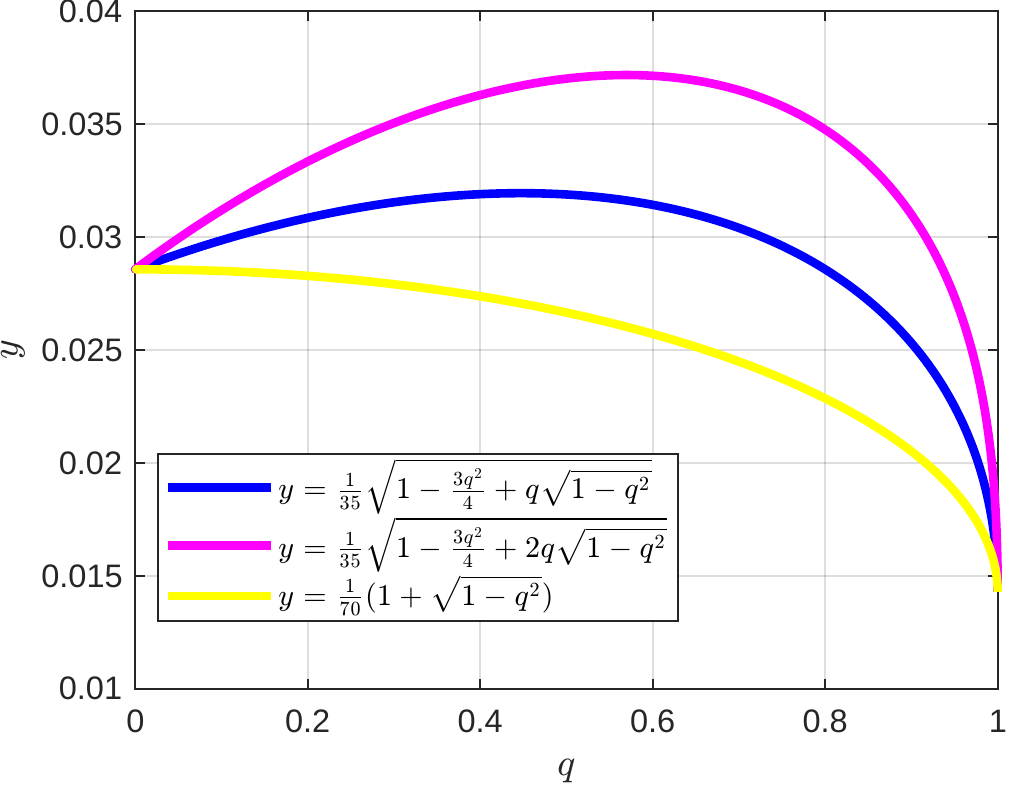}
\label{fig1}
\end{figure}

Again using $\|T^*T+TT^*\|=\frac{1}{35^2}$, from eqations (\ref{oth3}) and (\ref{eq7}) we get the upper bounds 

\begin{equation}\label{eq19}
    \omega_q(T)\leq \frac{q}{35\sqrt{2}(1-\sqrt{1-q^2})}
\end{equation}
and
\begin{equation}\label{eq20}
    \omega_q(T)\leq \frac{1}{35}\sqrt{1-\frac{q^2}{2}+q\sqrt{1-q^2}}.
\end{equation}

Figure \ref{fig2} presents the graphical representation of upper bounds (\ref{eq19}), (\ref{eq20}) and $\omega_q(T)$.

\begin{figure}[h]
\caption{Comparision of $\omega_q(T)$ with upper bounds (\ref{eq19}) and (\ref{eq20}) for Example \ref{example1}}
\centering
\includegraphics[scale=.90]{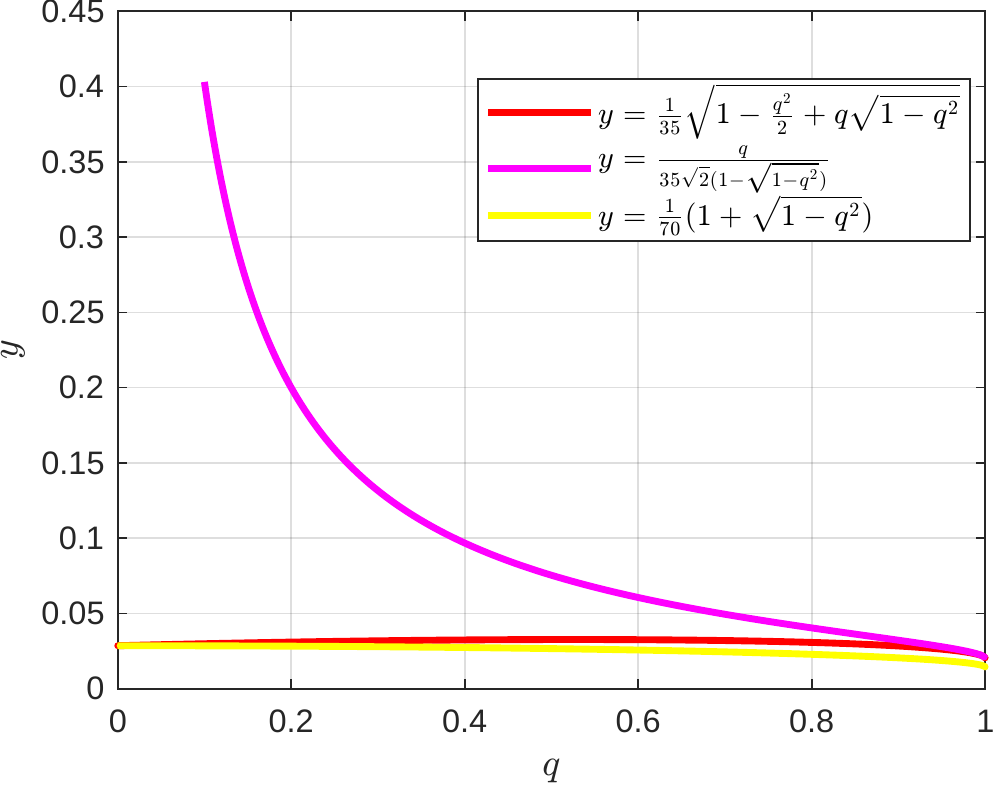}
\label{fig2}
\end{figure}

It is evident from Figure \ref{fig1} and \ref{fig2} that upper bounds obtained using Corollary \ref{cor2} and Corollary \ref{cor1} are more refined than those of \cite{fakhri2022q}. Figure \ref{fig3} shows the comparison between the obtained results (\ref{eq18}) and (\ref{eq20}). In this example, the upper bound obtained using Corollary \ref{cor2} gives a better result.

\begin{figure}[h]
\caption{Comparision of upper bounds (\ref{eq18}) and (\ref{eq20}) of $\omega_q(T)$ for Example \ref{example1}}
\centering
\includegraphics[scale=.90]{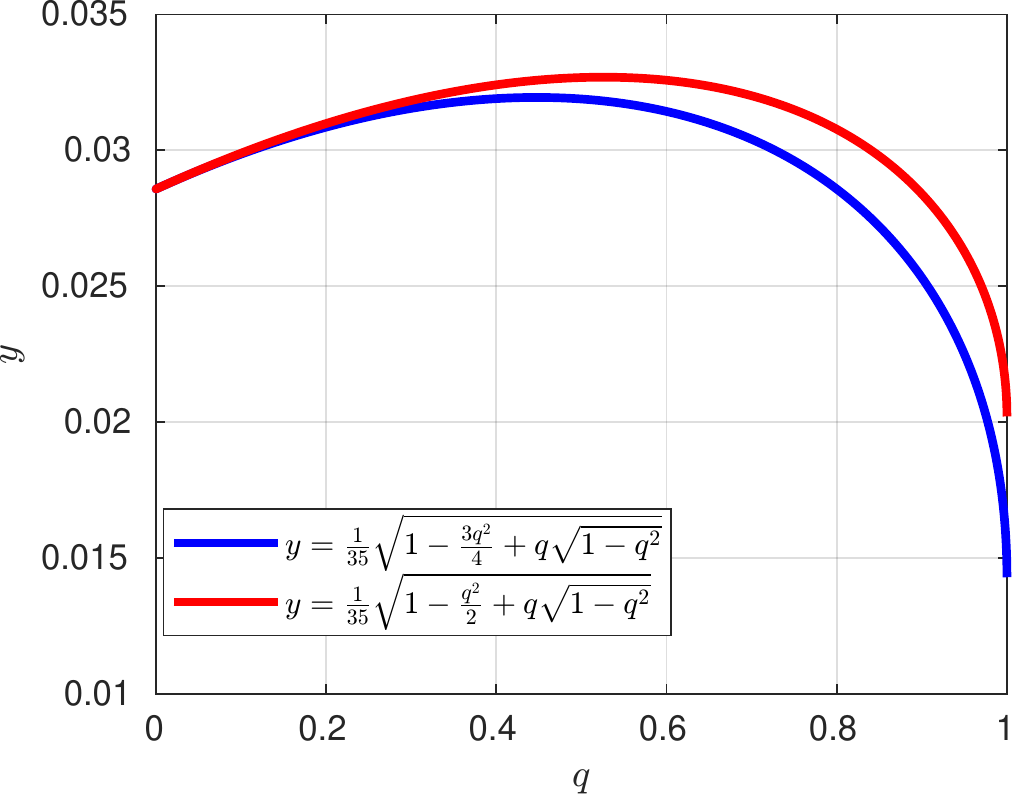}
\label{fig3}
\end{figure}

\end{example}

\begin{example}\label{example2}
Consider another matrix $T=\begin{pmatrix}0 & \frac{1}{25}\\ \frac{1}{36} & 0\end{pmatrix}$. Then $\|T\|=\frac{1}{25}$.
From Lemma \ref{lemma4}, 
\begin{equation*}
    W_q(T)=\left\{\frac{r}{1800}\left((61+\sqrt{1-q^2}11)+i(11+\sqrt{1-q^2}61)\right):0\leq r \leq 1,0\leq s \leq 2\pi)\right\}.
\end{equation*}
Therefore, $\omega_q(T)=\frac{1}{1800}(61+\sqrt{1-q^2}11)$. In this example also we will compare the upper bounds (\ref{oth2}) and (\ref{oth3}) of $\omega_q(T)$ obtained in \cite{fakhri2022q} with our results Corollary \ref{cor2} and Corollary \ref{cor1}. 
Since $\|T^2\|=\frac{1}{900}$, from equations (\ref{oth2}) and (\ref{eq5}) we repectively get 

\begin{eqnarray}\label{eq21} 
    \omega_q(T) \leq \frac{1}{25}\sqrt{1-\frac{23q^2}{144}+2q\sqrt{1-q^2}}
\end{eqnarray} and

\begin{eqnarray}\label{eq22}
    \omega_q(T) \leq \frac{1}{25}\sqrt{1-\frac{23q^2}{144}+q\sqrt{1-q^2}}.
\end{eqnarray}
Figure \ref{fig4} presents the graphical representation of upper bounds (\ref{eq21}), (\ref{eq22}) and $\omega_q(T)$. 

\begin{figure}[h]
\caption{Comparision of $\omega_q(T)$ with upper bounds (\ref{eq21}) and (\ref{eq22}) for Example \ref{example2}}
\centering
\includegraphics[scale=.90]{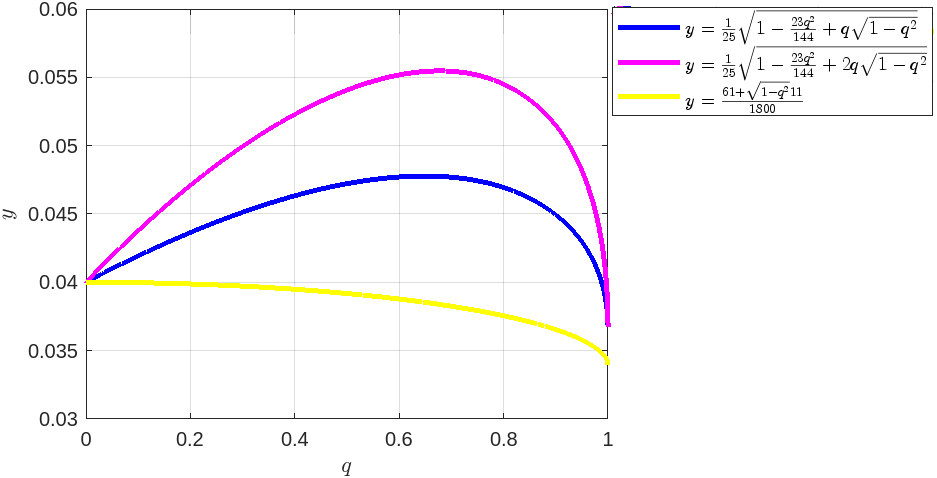}
\label{fig4}
\end{figure}

Now using $\|T^*T+TT^*\|=\frac{1921}{25^2 \times 36^2}$, from eqations (\ref{oth3}) and (\ref{eq7}) we get the upper bounds 

\begin{equation}\label{eq23}
    \omega_q(T)\leq \sqrt{\frac{1921}{2}}\frac{q}{900(1-\sqrt{1-q^2})}
\end{equation}
and
\begin{equation}\label{eq24}
    \omega_q(T)\leq \frac{1}{25}\sqrt{1-\frac{671q^2}{2592}+q\sqrt{1-q^2}}.
\end{equation}

Figure \ref{fig5} presents the graphical representation of upper bounds (\ref{eq23}), (\ref{eq24}) and $\omega_q(T)$.

\begin{figure}[h]
\caption{Comparision of $\omega_q(T)$ with upper bounds (\ref{eq23}) and (\ref{eq24}) for Example \ref{example2}}
\centering
\includegraphics[scale=1]{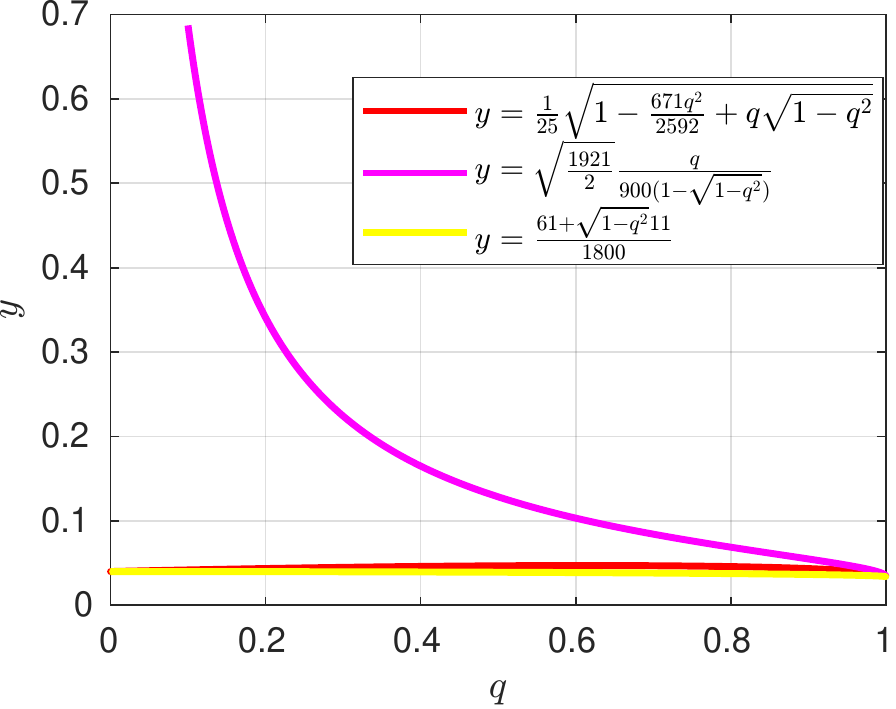}
\label{fig5}
\end{figure}

Again, it is clear from Figure \ref{fig4} and \ref{fig5} that upper bounds obtained using Corollary \ref{cor2} and Corollary \ref{cor1} are more refined than those of \cite{fakhri2022q}. Figure \ref{fig6} shows the comparison of our results (\ref{eq22}) and (\ref{eq24}). Contrary to Example \ref{example1}, in this example, the upper bound obtained using Corollary \ref{cor1} gives a better result than that of Corollary \ref{cor2}.

\begin{figure}[h]
\caption{Comparision of upper bounds (\ref{eq22}) and (\ref{eq24}) of $\omega_q(T)$ for Example \ref{example2}}
\centering
\includegraphics[scale=.95]{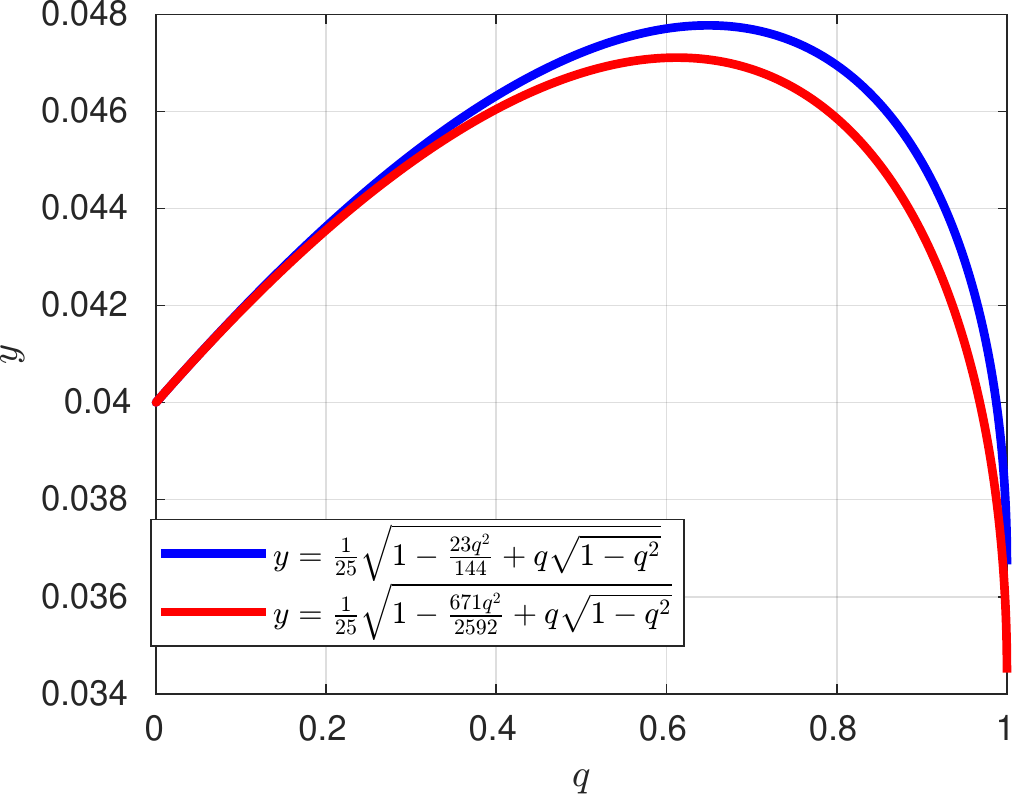}
\label{fig6}
\end{figure}

\end{example}

\begin{remark}
   Figure \ref{fig3} and Figure \ref{fig6} show that the upper bounds obtained in Corollary \ref{cor2} and  \ref{cor1} are noncomparable.
\end{remark}

Now we derive another upper bound of the $q$-numerical radius which refines the Theorem \ref{th2} and also the corresponding corollaries \ref{cor2} and \ref{cor1}.

\begin{theorem} \label{th5}
    Let $T \in \mathcal{B(H)}$ and $q\in [0,1],$ then
\begin{equation} \label{eq9}
    \omega_q^2(T) \leq q^2 \omega^2(T) + ( 1 - q^2 + q\sqrt{1 - q^2}) \|T\|^2 - (1 - q^2) c(T).
\end{equation}
\end{theorem}

\begin{proof}
  For $q=1$ the inequality holds trivially. Let $q \in [0,1).$ In accordance with the proof of Theorem \ref{th2}, it follows from the relation (\ref{array1})
    \begin{eqnarray*}
     |\langle Tx, y \rangle | & \leq & q |\langle Tx, x \rangle| + \sqrt{1 - q^2} (\|Tx\|^2 - |\Tx|^2)^{\frac{1}{2}}\\
    \Rightarrow  |\Txy|^2  & \leq & \left( q |\langle Tx, x \rangle| + \sqrt{1 - q^2} (\|Tx\|^2 - |\Tx|^2)^{\frac{1}{2}}  \right)^2 \nonumber\\
     & = & q^2 |\Tx|^2 + 2q \q |\Tx| (\|Tx\|^2 - |\Tx|^2)^{\frac{1}{2}} \nonumber \\
     && + (1 - q^2) (\|Tx\|^2 - |\Tx|^2) \nonumber \\
     & \leq & q^2 |\Tx|^2 + (1 - q^2 + q \q) \|Tx\|^2 - (1 - q^2) |\Tx|^2 \nonumber \\
     & \leq & q^2 \omega^2(T) + ( 1 - q^2 + q\sqrt{1 - q^2}) \|T\|^2 - (1 - q^2) c^2(T).
  \end{eqnarray*}
  Hence the required relation follows.
\end{proof}

\begin{remark}
Here we mention a few important observations related to the above theorem.
\begin{enumerate}
    \item[(i)] Clearly the relation (\ref{eq9}) improves the relation (\ref{eq3}) of Theorem \ref{th2} when $c(T)>0$. If $c(T) = 0,$ then the the upper bound in (\ref{eq9}) reduces to the upper bound of (\ref{eq3}).

    \item[(ii)] In this regard it is worth mentioning that $c(T) > 0$ if and only if $0 \notin \overline{W(T)}.$
\end{enumerate}
\end{remark}

The Corollary \ref{cor2} and Corollary \ref{cor1} can be improved by using the Theorem \ref{th5} as follows.

\begin{corollary} \label{cor3}
    For $T \in \mathcal{B(H)}$ and $q \in [0,1]$ we have
\begin{equation*} 
    \omega_q^2(T) \leq \frac{q^2}{4} (\|T\| + \|T^2\|^\frac{1}{2})^2 + ( 1 - q^2 + q\sqrt{1 - q^2}) \|T\|^2 - (1 - q^2) c^2(T).
\end{equation*}
\end{corollary}

\begin{corollary} \label{cor4}
    For $T \in \mathcal{B(H)}$ and $q \in [0,1],$ the following relation holds
   \begin{eqnarray*}
       \omega_q^2(T) \leq \frac{q^2}{2} \|T^*T+TT^*\| + (1 - q^2 + q \q) \|T\|^2 - (1 - q^2) c^2(T). 
    \end{eqnarray*} 
\end{corollary}

The following corollary relates the $q$-numerical radius, numerical radius, and the transcendental radius of an operator $T \in \mathcal{B(H)}.$

\begin{corollary} \label{th1}
Let $T \in \mathcal{B(H)}$ and $q\in [0,1],$ then
\begin{equation} \label{eq0}
    \omega_q(T) \leq q \omega(T) + \sqrt{1 - q^2} m(T).
\end{equation}
\end{corollary}

\begin{proof}
Since the case $q = 1$ is obvious, let $q \in [0,1)$ and $x, y \in \mathcal{H}$ such that $\|x\| = 1 = \|y\|$ with $\langle x,y \rangle = q.$ Then $y$ can be expressed as $y = q x + \sqrt{1 - q^2}z,$ where $\|z\| = 1$ and $\la x,z \ra = 0.$ In this setting we have
\begin{equation} \label{eq111}
 |\langle Tx, y \rangle | \leq q |\langle Tx, x \rangle| + \sqrt{1 - q^2}|\langle Tx, z \rangle | .
\end{equation}
Using Bessel's inequality, similar to Theorem \ref{th2}, we have
\[|\Txz|^2 \leq \|Tx\|^2 - |\la Tx, x \ra|^2.\]
Using the above relation in equation (\ref{eq111}), we get
\begin{eqnarray} \label{eq2}
  |\Txy | & \leq & q |\Tx| + \sqrt{1 - q^2} (\|Tx\|^2 - |\Tx|^2)^{\frac{1}{2}}\\
   & \leq & q \omega(T) + \q m(T). \nonumber
\end{eqnarray} 
Taking supremum over all $x$ and $y$ where $\|x\| = 1,$ $\|y\| = 1,$ $\la x,y \ra = q$ we get
\begin{equation*}
    \omega_q(T) \leq q \omega(T) + \q m(T).
\end{equation*}
This proves the result.
\end{proof}

Now we provide a more general $q$-numerical radius inequality from which several other inequalities related to product and commutators of operators follows. This result is $q$-numerical radius version of the inequality (\ref{eq8}). 

\begin{theorem}
If $A, B, C, D, S, T \in \mathcal{B(H)}$ and $q \in [0,1],$ then
    \begin{eqnarray*}
        \omega_q(ATB+CSD) &\leq & \frac{q}{2}\|B^*|T|^{2\alpha}B+A|T^*|^{2(1-\alpha)}A^*+D^*|S|^{2\alpha}D+C|S^*|^{2(1-\alpha)}C^*\|\\
       & & + \left(\sqrt{1-q^2}+\sqrt{2q\sqrt{1-q^2}}\right)\left(\sqrt{\|B^*|T|^{2\alpha}B\| \|A|T^*|^{2(1-\alpha)}A^*\|}+ \right. \\
       && \left. \sqrt{\|D^*|S|^{2\alpha}D^*\| \|C|S^*|^{2(1-\alpha)}C^*\|}\right).
    \end{eqnarray*}
\end{theorem}

\begin{proof}
The case $q=1$ follows directly from the relation (\ref{eq8}) which was derived in \cite{kittaneh2005numerical}. Let $x,y\in \mathcal{H}$ such that $\|x\| = \|y\| = 1$ with $\la x,y \ra = q.$ Then we have $y = qx + \q z$ where $\|z\| = 1$ and $\la x,z \ra = 0.$ Then from Lemma \ref{lemma1} we have
    \begin{eqnarray*}
        &&  | \la (ATB+CSD)x,y\ra | \\
        & \leq &  | \la TBx,A^*y\ra |+| \la SDx,C^*y\ra | \\
        & \leq & \la B^*|T|^{2\alpha}Bx,x\ra^{\frac{1}{2}}\la A|T^*|^{2(1-\alpha)}A^*y,y\ra^{\frac{1}{2}} 
        + \la D^*|S|^{2\alpha}Dx,x\ra^{\frac{1}{2}}\la C|S^*|^{2(1-\alpha)}C^*y,y\ra^{\frac{1}{2}} \\
        & \leq &  \la B^*|T|^{2\alpha}Bx,x\ra^{\frac{1}{2}}\left(q^2\la A|T^*|^{2(1-\alpha)}A^*x,x\ra+(1-q^2)\la A|T^*|^{2(1-\alpha)}A^*z,z\ra \right.\\
        & & \left. + 2q\sqrt{1-q^2} \left|\la A|T^*|^{2(1-\alpha)}A^*x,z\ra \right| \right)^{\frac{1}{2}} \\
        & & + \la D^*|S|^{2\alpha}Dx,x\ra^{\frac{1}{2}}\left(q^2\la C|S^*|^{2(1-\alpha)}C^*x,x\ra+(1-q^2)\la C|S^*|^{2(1-\alpha)}C^*z,z\ra \right.\\
        & & \left. + 2q\sqrt{1-q^2} \left|\la C|S^*|^{2(1-\alpha)}C^*x,z\ra \right| \right)^{\frac{1}{2}} \\
        & \leq & \la B^*|T|^{2\alpha}Bx,x\ra^{\frac{1}{2}}\left(q\la A|T^*|^{2(1-\alpha)}A^*x,x\ra^{\frac{1}{2}}+\sqrt{1-q^2}\la A|T^*|^{2(1-\alpha)}A^*z,z\ra^{\frac{1}{2}} \right.\\
        & & \left. + \sqrt{2q\sqrt{1-q^2}} \left|\la A|T^*|^{2(1-\alpha)}A^*x,z\ra \right|^{\frac{1}{2}} \right) \\
        & & + \la D^*|S|^{2\alpha}Dx,x\ra^{\frac{1}{2}}\left(q\la C|S^*|^{2(1-\alpha)}C^*x,x\ra^{\frac{1}{2}}+\sqrt{1-q^2}\la C|S^*|^{2(1-\alpha)}C^*z,z\ra^{\frac{1}{2}} \right.\\
        & & \left. + \sqrt{2q\sqrt{1-q^2}} \left|\la C|S^*|^{2(1-\alpha)}C^*x,z\ra \right|^{\frac{1}{2}} \right) \\
        & = & q \left(\la B^*|T|^{2\alpha}Bx,x\ra^{\frac{1}{2}}\la A|T^*|^{2(1-\alpha)}A^*x,x\ra^{\frac{1}{2}} +\la D^*|S|^{2\alpha}Dx,x\ra^{\frac{1}{2}}\la C|S^*|^{2(1-\alpha)}C^*x,x\ra^{\frac{1}{2}}\right) \\
        & & + \sqrt{1-q^2} \left(\la B^*|T|^{2\alpha}Bx,x\ra^{\frac{1}{2}}\la A|T^*|^{2(1-\alpha)}A^*z,z\ra^{\frac{1}{2}} +\la D^*|S|^{2\alpha}Dx,x\ra^{\frac{1}{2}} \right. \times \\
        & & \left. \la C|S^*|^{2(1-\alpha)}C^*z,z\ra^{\frac{1}{2}}\right) + \sqrt{2q\sqrt{1-q^2}} \left(\la B^*|T|^{2\alpha}Bx,x\ra^{\frac{1}{2}} \left|\la A|T^*|^{2(1-\alpha)}A^*x,z\ra\right|^{\frac{1}{2}} \right. \\
        && \left. +\la D^*|S|^{2\alpha}Dx,x\ra^{\frac{1}{2}}\left| \la C|S^*|^{2(1-\alpha)}C^*x,z\ra\right|^{\frac{1}{2}}\right)\\
        & \leq & \frac{q}{2} \la \left(B^*|T|^{2\alpha}B+A|T^*|^{2(1-\alpha)}A^*+D^*|S|^{2\alpha}D+C|S^*|^{2(1-\alpha)}C^*\right)x,x\ra \\
        & & + \sqrt{1-q^2} \left(\la B^*|T|^{2\alpha}Bx,x\ra^{\frac{1}{2}}\la A|T^*|^{2(1-\alpha)}A^*z,z\ra^{\frac{1}{2}} +\la D^*|S|^{2\alpha}Dx,x\ra^{\frac{1}{2}} \right. \times \\
        & & \left. \la C|S^*|^{2(1-\alpha)}C^*z,z\ra^{\frac{1}{2}}\right) + \sqrt{2q\sqrt{1-q^2}} \left(\la B^*|T|^{2\alpha}Bx,x\ra^{\frac{1}{2}}\left|\la A|T^*|^{2(1-\alpha)}A^*x,z\ra\right|^{\frac{1}{2}} \right.\\
        && \left. +\la D^*|S|^{2\alpha}Dx,x\ra^{\frac{1}{2}}\left|\la C|S^*|^{2(1-\alpha)}C^*x,z\ra\right|^{\frac{1}{2}}\right)\\
        & \leq & \frac{q}{2} \|B^*|T|^{2\alpha}B+A|T^*|^{2(1-\alpha)}A^*+D^*|S|^{2\alpha}D+C|S^*|^{2(1-\alpha)}C^*\|+\\
        & & \left(\sqrt{1-q^2}+\sqrt{2q\sqrt{1-q^2}}\right)\left(\sqrt{\|B^*|T|^{2\alpha}B\| \|A|T^*|^{2(1-\alpha)}A^*\|}+\sqrt{\|D^*|S|^{2\alpha}D\| \|C|S^*|^{2(1-\alpha)}C^*\|}\right).
    \end{eqnarray*}
    The result follows by taking supremum over all such $x,y \in \mathcal{H}$ with $\|x\| = \|y\| = 1,$ and $\la x,y \ra = q$ on the left-hand side.
\end{proof}
Here we mention a few particular cases of the above theorem.
\begin{remark}
    Consider a few particular cases of the above result:
    \begin{enumerate}
        \item[(i)] If $T=B=I$ and $S=0$ then 
        \begin{equation*}
            \omega_q(A)\leq\frac{q}{2}\|AA^*+I\|+\left(\sqrt{1-q^2}+\sqrt{2q\sqrt{1-q^2}}\right)\|A\|.
        \end{equation*}
        Also, if $A=B=I$, $S = 0$ and $\alpha = \frac{1}{2}$ then we have
        \begin{eqnarray*}
            \omega_q(T) & \leq & \frac{q}{2}\||T| + |T^*|\|+\left(\sqrt{1-q^2}+\sqrt{2q\sqrt{1-q^2}}\right)\sqrt{\||T|\| \||T^*|\|}\\
            & = & \frac{q}{2}\||T| + |T^*|\|+\left(\sqrt{1-q^2}+\sqrt{2q\sqrt{1-q^2}}\right) \|T\|.
        \end{eqnarray*}
        
        \item[(ii)] If $T=I$ and $S=0$ then 
        \begin{equation*}
          \omega_q(AB)\leq\frac{q}{2}\|AA^*+BB^*\|+\left(\sqrt{1-q^2}+\sqrt{2q\sqrt{1-q^2}}\right)\|A\|\|B\|.  
        \end{equation*}
    
        \item[(iii)] If $T=I$, $C=B$ and $D=A$ then
        \begin{equation*}
            \omega_q(AB+BA)\leq\frac{q}{2}\|AA^*+B^*B+A^*A+BB^*\|+\left(\sqrt{1-q^2}+\sqrt{2q\sqrt{1-q^2}}\right)\|A\|^2\|B\|^2.
        \end{equation*}
    \end{enumerate}
\end{remark}

Now we focus on some bounds of $q$-numerical radius which are not dependent on $q$. For the following theorem, we use a similar concept of Theorem 2 in \cite{kittaneh2005numerical}.

\begin{theorem}
    If $A, B, C, D, S, T \in \mathcal{B(H)}$ and $q \in [0,1],$ then
    \begin{equation*}
        \omega_q(ATB + CSD) \leq \frac{1}{2} \left( \|A|T^*|^{2(1 - \alpha)} A^* + C |S^*|^{2(1 - \alpha)} C^* \|  + \| B^* |T|^{2 \alpha} B + D^* |S|^{2 \alpha} D \| \right),
    \end{equation*}
    holds for all $\alpha \in [0,1].$
\end{theorem}

\begin{proof}
    Let $x,y \in \mathcal{H}$ with $\|x\| = \|y\| =1$ and $\la x,y \ra = q.$ Then from Lemma \ref{lemma1} and AM-GM inequalty we have
    \begin{eqnarray*}
        |\la (ATB + CSD)x,y \ra | & \leq & |\la TBx,A^*y \ra| + |\la SDx,C^*y \ra| \\
        & \leq & \la |T|^{2 \alpha} Bx, Bx \ra^{\frac{1}{2}} \la |T^*|^{2(1- \alpha)} A^*y, A^*y \ra^{\frac{1}{2}}\\
        && + \la |S|^{2 \alpha} Dx, Dx \ra^{\frac{1}{2}} \la |S^*|^{2(1- \alpha)} C^*y, C^*y \ra^{\frac{1}{2}}\\
        & \leq & \frac{1}{2} \left(\la |T|^{2 \alpha} Bx, Bx \ra +  \la |T^*|^{2(1- \alpha)} A^*y, A^*y \ra \right)\\
        && + \frac{1}{2} \left(\la |S|^{2 \alpha} Dx, Dx \ra +  \la |S^*|^{2(1- \alpha)} C^*y, C^*y \ra \right)\\
        & = & \frac{1}{2} \la \left(A|T^*|^{2(1- \alpha)} A^* +  C|S^*|^{2(1- \alpha)} C^*\right)y, y \ra \\
        && + \frac{1}{2} \la \left(B^*|T|^{2 \alpha} B + D^*|S|^{2 \alpha}D\right)x, x \ra \\
        & \leq & \frac{1}{2} \|A|T^*|^{2(1 - \alpha)} A^* + C |S^*|^{2(1 - \alpha)} C^* \|  + \frac{1}{2} \| B^* |T|^{2 \alpha} B + D^* |S|^{2 \alpha} D \|.
    \end{eqnarray*}
    The required result follows by taking supremum over all $x,y \in \mathcal{H}$ with $\|x\| = \|y\| =1$ and $\la x,y \ra = q$ on the left-hand side.
\end{proof}

\begin{remark}
\begin{enumerate}
    \item[(i)] In particular, let $A = B = I$ and $S = 0$ then we have
    \begin{equation*}
        \omega_q(T) \leq \frac{1}{2} \left( \| |T|^{2 \alpha} \| + \| |T^*|^{2(1- \alpha)} \| \right), \ \alpha \in [0,1].
    \end{equation*}
    By choosing $\alpha = \frac{1}{2}$ and using $\| |T| \| = \||T^*| \| = \|T\|,$ we get
    \[\omega_q(T) \leq \frac{1}{2}\left( \| |T| \| + \| |T^*| \| \right) = \|T\|. \]

    \item[(ii)] If $T=I$ and $S=0$ then we get
    \[\omega_q(AB) \leq \frac{1}{2} (\|AA^*\| + \|B^*B\|).\]
    Also by putting $A = I,$ $T = A$ and $S = 0$ we have
    \[\omega_q(AB) \leq \frac{1}{2} (\| |A^*|^{2(1- \alpha)} \| + \| B^* |A|^{2 \alpha} B\|),\]
    and by choosing $\alpha = \frac{1}{2}$ it follows
    \[\omega_q(AB) \leq \frac{1}{2} (\| A \| + \| B^* |A| B\|).\]
\end{enumerate}
    
\end{remark}

\section{$q$-numerical radius of $2 \times 2$ operator matrices}

In this section, we derive a few bounds for the $q$-numerical range of $2 \times 2$ operator matrices. Let $\h$ and $\hh$ are two complex Hilbert spaces with the inner product $\la.,. \ra$. Then $\hhh$ forms a Hilbert space and any operator $T \in \mathcal{B}({\hhh})$ has an $2 \times 2$ matrix representation of the form
\[T = \begin{bmatrix} A & B\\
C &D \end{bmatrix},\]
where $A \in \mathcal{B}(\h), B \in \mathcal{B}(\hh, \h), C \in \mathcal{B}(\h, \hh), \ \mbox{and} \ D \in \mathcal{B}(\hh).$

 As the $q$-numerical radius forms a weakly unitarily invariant norm, from the result mentioned in (P. 107 \cite{bhatia2013matrix}), we can deduce the following inequalities
\begin{equation} \label{eq10}
    \omega_q \left( \begin{bmatrix} A & 0\\
0 &D \end{bmatrix} \right) \leq \omega_q \left( \begin{bmatrix} A & B\\
C &D \end{bmatrix} \right)
\end{equation}
and
\begin{equation} \label{eq11}
 \omega_q \left( \begin{bmatrix} 0 & B\\
C & 0 \end{bmatrix} \right) \leq \omega_q \left( \begin{bmatrix} A & B\\
C &D \end{bmatrix} \right).   
\end{equation}

The numerical radius of diagonal operator matrices enjoy the following equality (\cite{hirzallah2012numerical})
\begin{equation} \label{eq25}
   \omega \left( \begin{bmatrix} A & 0\\
0 & D \end{bmatrix} \right) = \max \{\omega(A), \omega(D)\}. 
\end{equation}
A similar conclusion is not true for $q$-numerical radius. This is seen in the Example \ref{example3} that follows. The following lemma on the $q$-numerical radius of Hermitian matrices is required for this. 

\begin{lemma}[Theorem 3.5 \cite{gau2021numerical}] \label{lemma5}
If $T$ is an $n \times n$ Hermitian matrix with eigenvalues $\lambda_1 \geq \lambda_2 \geq \cdots \geq \lambda_n$ and $|q|\leq 1$, then the numerical range $W_q(T)$ equals the (closed) elliptic disc with foci $q\lambda_1$ and $q \lambda_n$ and minor axis of length $\sqrt{1 - |q|^2}(\lambda_1 - \lambda_n).$ 
\end{lemma}

\begin{example}\label{example3}

Let $q \in [0,1]$. For any $2 \times 2$ matrix $T=\begin{bmatrix} a & 0\\
0 & d \end{bmatrix}$ with $a,d > 0$ and $a \neq d$, the $q$-numerical range of $T$ is given by
\begin{equation}
    W_q(T) = \left\lbrace (x,y) : \frac{\left( x - \frac{q}{2} (a+d)\right)^2}{\frac{1}{4}(a-d)^2} + \frac{y^2}{\frac{1}{4}(1 - q^2)(a-d)^2} \leq 1 \right\rbrace. 
\end{equation}
The $q$-numerical radius of $T=\begin{bmatrix} a & 0\\
0 & d \end{bmatrix}$ is given by the following maximization problem
\[\max \sqrt{x^2+y^2}, \ \mbox{subject to } \frac{\left( x - \frac{q}{2} (a+d)\right)^2}{\frac{1}{4}(a-d)^2} + \frac{y^2}{\frac{1}{4}(1 - q^2)(a-d)^2} = 1.\]
Solving this we get 
\[\max \sqrt{x^2 + y^2} = \left( \frac{q}{2}(a+d) + \frac{1}{2}|a-d| \right) \ \mbox{occurs at } \left( \frac{q}{2}(a+d) + \frac{1}{2}|a-d|, 0 \right). \]
Hence \[\omega_q\left( \begin{bmatrix} a & 0\\
0 & d \end{bmatrix} \right) =\left( \frac{q}{2}(a+d) + \frac{1}{2}|a-d| \right) > q \max \{a,d \} \ \mbox{when } q \neq 1. \]
The above equation implies that when $q \neq 1,$ the $q$-numerical range does not satisfy the relation
\[\omega_q \left( \begin{bmatrix} A & 0\\
0 & D \end{bmatrix} \right) = \max \{\omega(A)_q, \omega_q(D)\}\]
similar to the relation mentioned in (\ref{eq25}).

\end{example}

In this regard, the following result provides the upper and lower bound of the $q$-numerical radius.

\begin{theorem} \label{th6}
    Let $\h$ and $\hh$ are Hilbert spaces and let $A \in \mathcal{B}(\h),$ $B \in \mathcal{B}(\hh,\h),$ $C \in \mathcal{B}(\h, \hh),$ $D \in \mathcal{B}(\hh)$ and $q \in [0,1].$ Then the following inequalities hold
    \begin{enumerate}
        \item[(i)] $ \max \left\lbrace \omega_q(A), \omega_q(D), \omega_q \left( \begin{bmatrix} 0 & B\\
C & 0 \end{bmatrix} \right)\right\rbrace \leq \omega_q \left( \begin{bmatrix} A & B\\
C &D \end{bmatrix} \right),$\\

\item[(ii)] $ \omega_q \left( \begin{bmatrix} A & B\\
C &D \end{bmatrix} \right) \leq \max \{\|A\|, \|D\|\} + \left(1 - \frac{3q^2}{4} + q \q \right)^{\frac{1}{2}} (\|B\| + \|C\|),$\\

\item[(iii)] $ \omega_q \left( \begin{bmatrix} A & B\\
C &D \end{bmatrix} \right) \leq \q \left(\|A\|^2 + \|B\|^2 + \|C\|^2 + \|D\|^2 \right)^{\frac{1}{2}}$ \\ 
\hspace*{2.7cm} $+ q \left( \max \{\omega(A), \omega(D)\} + \frac{\|B\| + \|C\|}{2}\right).$
    \end{enumerate}
\end{theorem}

\begin{proof}
\begin{enumerate}

    \item[(i)] The case $q=1$ follows directly from relation (\ref{eq25}). Let $q \in [0,1)$ and $x_1, y_1 \in \mathcal{H}_1$ such that $\|x_1\| = \|y_1\| = 1$ with $\la x_1, y_1 \ra = q$. Then
    \begin{eqnarray*}
        \omega_q \left( \begin{bmatrix} A & B\\
C &D \end{bmatrix} \right) & \geq & \left| \la \begin{bmatrix} A & B\\
C &D \end{bmatrix} \begin{pmatrix} x_1  \\ 0 \end{pmatrix}, \begin{pmatrix} y_1  \\ 0 \end{pmatrix}  \ra  \right|\\
 & = & |\la Ax_1, y_1 \ra|.
    \end{eqnarray*}
    Taking supremum over all such $x_1$ and $y_1$ with $\|x_1\| = \|y_1\| = 1$ and $\la x_1, y_1 \ra = q,$ it follows that
    \begin{equation} \label{eq12}
        \omega_q \left( \begin{bmatrix} A & B\\
C &D \end{bmatrix} \right) \geq \omega_q(A).
    \end{equation}
    In a similar way, it can be proved that 
    \begin{equation} \label{eq13}
        \omega_q \left( \begin{bmatrix} A & B\\
C &D \end{bmatrix} \right) \geq \omega_q(D).
    \end{equation}
    Finally from the relations (\ref{eq10}), (\ref{eq12}), and (\ref{eq13}) we get 
    \[ \max \left\lbrace \omega_q(A), \omega_q(D), \omega_q \left( \begin{bmatrix} 0 & B\\
C & 0 \end{bmatrix} \right)\right\rbrace \leq \omega_q \left( \begin{bmatrix} A & B\\
C &D \end{bmatrix} \right).\]

\item[(ii)] Note that the $q$-numeical range forms a semi-norm and the following relation holds
\begin{equation} \label{eq14}
   \omega_q \left( \begin{bmatrix} A & B\\
C &D \end{bmatrix} \right) \leq \omega_q \left( \begin{bmatrix} A & 0\\
0 &D \end{bmatrix} \right) + \omega_q \left( \begin{bmatrix} 0 & B\\
0 & 0 \end{bmatrix} \right) + \omega_q \left( \begin{bmatrix} 0 & 0\\
C &0 \end{bmatrix} \right).
\end{equation}
Since $\begin{bmatrix} 0 & B\\
0 & 0 \end{bmatrix}^2 = \begin{bmatrix} 0 & 0\\
C & 0 \end{bmatrix}^2 = \begin{bmatrix} 0 & 0\\
0 & 0 \end{bmatrix},$ Theorem 2.5 of \cite{fakhri2022q} and (\ref{eq14}) imply that

\begin{equation} \label{eq15}
    \omega_q \left( \begin{bmatrix} A & B\\
C & D \end{bmatrix} \right) \leq \omega_q \left( \begin{bmatrix} A & 0\\
0 & D \end{bmatrix} \right) + \left(1 - \frac{3q^2}{4} + q \q \right)^{\frac{1}{2}} (\|B\| + \|C\|).
\end{equation}
Let $\begin{pmatrix} x_1  \\ x_2 \end{pmatrix} , \begin{pmatrix} y_1  \\ y_2 \end{pmatrix} \in \hhh,$ with
\[ \norm{\begin{pmatrix} x_1  \\ x_2 \end{pmatrix} } = \norm{\begin{pmatrix} y_1  \\ y_2 \end{pmatrix}} = 1, \ \mbox{and } \la \begin{pmatrix} x_1  \\ x_2 \end{pmatrix} , \begin{pmatrix} y_1  \\ y_2 \end{pmatrix} \ra = q. \]
Observe that
\begin{eqnarray*}
  \left| \la \begin{bmatrix} A & 0 \\
0 &D \end{bmatrix} \begin{pmatrix} x_1  \\ x_2 \end{pmatrix}, \begin{pmatrix} y_1  \\ y_2 \end{pmatrix}  \ra  \right| & \leq & |\la Ax_1, y_1 \ra| + |\la Dx_2, y_2 \ra| \\
& \leq & \|A\| \|x_1\| \|y_1\| + \|D\| \|x_2\| \|y_2\| \\
& \leq & \max \{\|A\|, \|D\|\}.
\end{eqnarray*}
Hence we have the following result.
\begin{equation} \label{eq16}
    \omega_q \left( \begin{bmatrix} A & 0\\
0 &D \end{bmatrix} \right) \leq \max \{\|A\|, \|D\|\}.
\end{equation}
The required upper bound of  $\omega_q \left( \begin{bmatrix} A & B\\
C & D \end{bmatrix} \right)$ follows from the inequalities (\ref{eq15}) and (\ref{eq16}).

\item[(iii)] To prove the last inequality let 
$\begin{pmatrix} x_1  \\ x_2 \end{pmatrix} , \begin{pmatrix} y_1  \\ y_2 \end{pmatrix} \in \hhh,$ with
\[ \norm{\begin{pmatrix} x_1  \\ x_2 \end{pmatrix} } = \norm{\begin{pmatrix} y_1  \\ y_2 \end{pmatrix}} = 1, \ \mbox{and } \la \begin{pmatrix} x_1  \\ x_2 \end{pmatrix} , \begin{pmatrix} y_1  \\ y_2 \end{pmatrix} \ra = q. \]
In this setting, we can take 
\[y_1 = qx_1 + \q z_1, \ \mbox{and } y_2 = qx_2 + \q z_2\]
where $z_1 \in \h, \ z_2 \in \hh$ with 
\[ \norm{\begin{pmatrix} z_1  \\ z_2 \end{pmatrix}} = 1, \ \mbox{and } \la \begin{pmatrix} x_1  \\ x_2 \end{pmatrix} , \begin{pmatrix} z_1  \\ z_2 \end{pmatrix} \ra = 0.\]
We assume that $\|x_1\| = \cos \theta, \ \|x_2\| = \sin \theta, \ \|z_1\| = \cos \phi, \ \mbox{and} \ \|z_2\| = \sin \phi$ where $\theta, \phi \in [0, \frac{\pi}{2}].$ Also, we use the fact that for any $a,b \in \mathbb{R},$
\begin{enumerate}
    \item[(1)] $\max\limits_\theta (a \cos \theta + b \sin \theta) = \sqrt{a^2 + b^2},$
    \item[(2)] $\max\limits_\theta (a \cos^2 \theta + b \sin^2 \theta) = \max \{a,b\}.$
\end{enumerate}
In this setting we have
\begin{equation} \label{eq26}
    \left| \la \begin{bmatrix} A & B\\
C &D \end{bmatrix} \begin{pmatrix} x_1  \\ x_2 \end{pmatrix}, \begin{pmatrix} y_1  \\ y_2 \end{pmatrix}  \ra  \right| \leq |\la Ax_1, y_1 \ra| + |\la Bx_2, y_1 \ra| + |\la Cx_1, y_2 \ra| + |\la Dx_2, y_2 \ra|.
\end{equation}
The subsequent computations are mentioned below.
\begin{eqnarray*}
     & & |\la Ax_1, y_1 \ra| + |\la Bx_2, y_1 \ra| + |\la Cx_1, y_2 \ra| + |\la Dx_2, y_2 \ra| \\
        & = & |\la Ax_1, qx_1 + \q z_1 \ra| + |\la Bx_2, qx_1 + \q z_1 \ra| \\
        && + |\la Cx_1, qx_2 + \q z_2 \ra| + |\la Dx_2, qx_2 + \q z_2 \ra| \\
        & \leq & q (|\la Ax_1, x_1 \ra| + |\la Bx_2, x_1 \ra| + |\la Cx_1, x_2 \ra| + |\la Dx_2, x_2 \ra|)\\
        && + \q \left(|\la Ax_1, z_1 \ra| + |\la Bx_2, z_1 \ra| + |\la Cx_1, z_2 \ra| + |\la Dx_2, z_2 \ra|\right) \\
        & \leq & q \left( \omega(A) \cos^2 \theta + \omega(D) \sin^2 \theta + \frac{1}{2}(\|B\|+\|C\|) \sin 2\theta \right) + \q \left( \|A\| \cos \theta \cos \phi \right.\\
        && \left. +  \|B\| \sin \theta \cos \phi + \|C\| \cos \theta \sin \phi + \|D\| \sin \theta \sin \phi    \right)\\
        & \leq & q \left( \max \{\omega(A), \omega(D)\} + \frac{\|B\| + \|C\|}{2}  \right) + \q \left( \|A\|^2 + \|B\|^2 + \|C\|^2 + \|D\|^2\right)^{\frac{1}{2}}.
\end{eqnarray*}
The required result follows from the above inequality and relation (\ref{eq26}).
\end{enumerate}
\end{proof}

\begin{remark}
    Below are a few highlights of the aforementioned theorem. 
    \begin{enumerate}
        \item[(i)] If $B=C=0$ in Theorem \ref{th6}(i) and (ii), it follows 
        \[\max \left\lbrace \omega_q(A), \omega_q(D)\right\rbrace \leq \omega_q \left( \begin{bmatrix} A & 0 \\
0 &D \end{bmatrix} \right) \leq \max\{\|A\|, \|D\|\}.\]

The above relation implies that $\max \left\lbrace \omega_q(A), \omega_q(D)\right\rbrace$ actually provides a lower bound of $\omega_q \left( \begin{bmatrix} A & 0 \\
0 &D \end{bmatrix} \right)$ where it is proved that $q$-numerical radius fails to satisfy an analogous relation as of equation (\ref{eq25}).

\item[(ii)] If $q = 1$ then first and third relations of Theorem \ref{th6} reduces to existing lower and upper bounds of $\omega \left( \begin{bmatrix} A & B \\
C &D \end{bmatrix} \right)$ mentioned in \cite{hirzallah2012numerical}.

\item[(iii)] The upper bounds obtained in (ii) and (iii) of Theorem \ref{th6} are noncomparable. Let $T=\begin{bmatrix}
    a & b \\
    c  & d
\end{bmatrix}=\begin{bmatrix}
    1.5442 + 1.4193i & 0.0859 + 0.2916i \\
    -1.4916 + 0.1978i & -0.7423 + 1.5877i
\end{bmatrix}$, randomly generated by MATLAB command \texttt{randn}. Figure \ref{figend}, demonstrates the comparison of the upper bounds (ii) and (iii) of Theorem \ref{th6} for the matrix $T$. 
  \end{enumerate}

\begin{figure}[h]
\caption{Comparision of upper bounds (ii) and (iii) of Theorem \ref{th6} for $\omega_q(T)$}
\centering
\includegraphics[scale=.40]{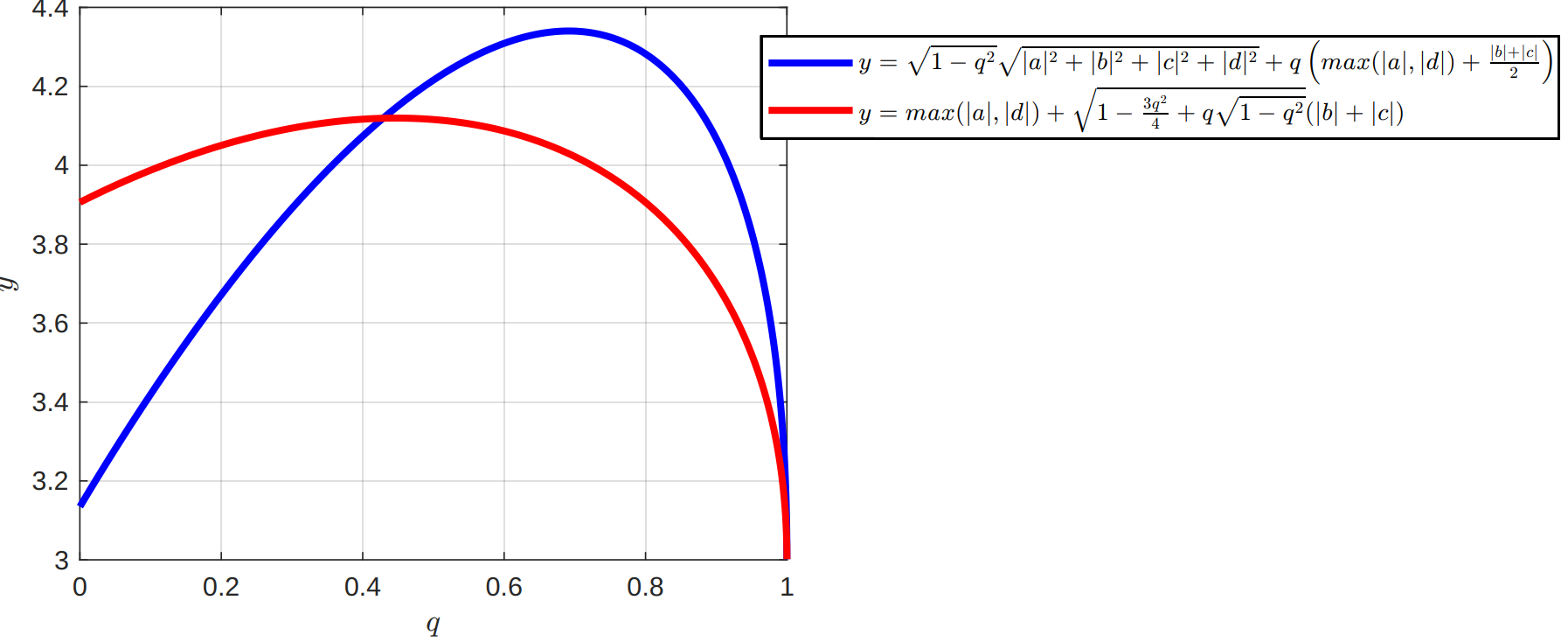}
\label{figend}
\end{figure}

\end{remark}

\section*{Disclosure statement}
No potential conflict of interest is reported by the authors.

\bibliographystyle{tfnlm}
\bibliography{interactnlmsample}

\begin{thebibliography}{10}
\providecommand{\url}[1]{\normalfont{#1}}
\providecommand{\urlprefix}{Available from: }

\bibitem{kittaneh2003numerical}
Kittaneh~F. A numerical radius inequality and an estimate for the numerical
  radius of the frobenius companion matrix. Studia Math.
  2003;\hspace{0pt}158(1):11--17.

\bibitem{kittaneh2005numerical}
Kittaneh~F. Numerical radius inequalities for {H}ilbert space operators. Studia
  Math. 2005;\hspace{0pt}168(1):73--80.

\bibitem{bhunia2022lectures}
Bhunia~P, Dragomir~SS, Moslehian~MS, et~al. Lectures on numerical radius
  inequalities. Infosys Science Foundation Series in Mathematical Sciences,
  Springer. 2022;\hspace{0pt}.

\bibitem{marcus1977constrained}
Marcus~M, Andresen~P. Constrained extrema of bilinear functionals. Monatsh
  Math. 1977;\hspace{0pt}84(3):219--235.

\bibitem{tsing1984constrained}
Tsing~NK. The constrained bilinear form and the c-numerical range. Linear
  Algebra Appl. 1984;\hspace{0pt}56:195--206.

\bibitem{li1994generalized}
Li~CK, Mehta~PP, Rodman~L. A generalized numerical range: the range of a
  constrained sesquilinear form. Linear Multilinear Algebra.
  1994;\hspace{0pt}37(1-3):25--49.

\bibitem{li1998some}
Li~CK, Nakazato~H. Some results on the q-numerical. Linear Multilinear Algebra.
  1998;\hspace{0pt}43(4):385--409.

\bibitem{chien2002davis}
Chien~MT, Nakazato~H. Davis--wielandt shell and q-numerical range. Linear
  Algebra Appl. 2002;\hspace{0pt}340(1-3):15--31.

\bibitem{chien2007q}
Chien~MT, Nakazato~H. The q-numerical radius of weighted shift operators with
  periodic weights. Linear Algebra Appl. 2007;\hspace{0pt}422(1):198--218.

\bibitem{chien2012numerical}
Chien~MT. The numerical radius of a weighted shift operator. RIMS
  K{\^o}ky{\^u}roku. 2012;\hspace{0pt}1778:70--77.

\bibitem{duan2009perfect}
Duan~R, Feng~Y, Ying~M. Perfect distinguishability of quantum operations. Phys
  Rev Lett. 2009;\hspace{0pt}103(21):210501.

\bibitem{fakhri2022q}
Fakhri~Moghaddam~S, Kamel~Mirmostafaee~A, Janfada~M. q-numerical radius
  inequalities for {H}ilbert space. Linear Multilinear Algebra.
  2022;\hspace{0pt}:1--13.

\bibitem{gau2021numerical}
Gau~HL, Wu~PY. Numerical ranges of hilbert space operators. Vol. 179. Cambridge
  University Press; 2021.

\bibitem{stampfli1970norm}
Stampfli~J. The norm of a derivation. Pac J Math.
  1970;\hspace{0pt}33(3):737--747.

\bibitem{prasanna1981norm}
Prasanna~S. The norm of a derivation and bjorck-thomee'lstratescu theorem. Math
  Japon. 1981;\hspace{0pt}26:585--588.

\bibitem{furuta1986simplified}
Furuta~T. A simplified proof of heinz inequality and scrutiny of its equality.
  Proc Amer Math Soc. 1986;\hspace{0pt}97(4):751--753.

\bibitem{nakazato1994c}
Nakazato~H. The c-numerical range of a 2$\times$ 2 matrix. Sci Rep Hirosaki
  Univ. 1994;\hspace{0pt}41:197--206.

\bibitem{bhatia2013matrix}
Bhatia~R. Matrix analysis. Vol. 169. Springer Science \& Business Media; 2013.

\bibitem{hirzallah2012numerical}
Hirzallah~O, Kittaneh~F, Shebrawi~K. Numerical radius inequalities for
  2$\times$ 2 operator matrices. Studia Math. 2012;\hspace{0pt}210(2):99--115.

\end{thebibliography}

\end{document}